\documentclass{amsart}
\usepackage{amsfonts}
\usepackage{mathrsfs}
\usepackage{amsthm}
\usepackage{amssymb}
\usepackage{amsmath}
\usepackage{enumerate}

\theoremstyle{plain}
\newtheorem{theorem}{Theorem}[section]

\newtheorem{corollary}[theorem]{Corollary}
\newtheorem{lemma}[theorem]{Lemma}

\theoremstyle{definition}
\newtheorem{definition}{Definition}[section]

\theoremstyle{remark}
\newtheorem{remark}{Remark}[section]

\theoremstyle{example}

\numberwithin{equation}{section}

\begin{document}

\title[Dirichlet Forms]
{Dirichlet Forms Constructed from Annihilation Operators on Bernoulli Functionals}

\author{Caishi Wang}
\address[Caishi Wang]
          {School of Mathematics and Statistics,
          Northwest Normal University,
          Lanzhou, Gansu 730070,
          People's Republic of China }
\email{wangcs@nwnu.edu.cn}

\author{Beiping Wang}
\address[Beiping Wang]
         {School of Mathematics and Statistics,
         Northwest Normal University,
          Lanzhou, Gansu 730070,
          People's Republic of China}


\begin{abstract}
The annihilation operators on Bernoulli functionals (Bernoulli annihilators, for short)
and their adjoint operators satisfy a canonical anti-commutation relation (CAR) in equal-time.
As a mathematical structure, Dirichlet forms play an important role in many fields in mathematical physics.
In this paper, we apply the Bernoulli annihilators to constructing Dirichlet forms on Bernoulli functionals.
Let $w$ be a nonnegative function on $\mathbb{N}$.  By using the Bernoulli annihilators,
we first define in a dense subspace of the $L^2$-space of Bernoulli functionals a positive, symmetric bilinear form $\mathcal{E}_w$ associated with $w$.
And then we prove that $\mathcal{E}_w$ is closed and has the contraction property, hence it is a Dirichlet form.
Finally, we consider an interesting semigroup of operators associated with $w$ on the $L^2$-space of Bernoulli functionals, which we call
the $w$-Ornstein-Uhlenbeck semigroup, and by using the Dirichlet form $\mathcal{E}_w$ we show
that the $w$-Ornstein-Uhlenbeck semigroup is a Markov semigroup.
\end{abstract}

\maketitle

\section{Introduction}\label{sec-1}

The annihilation operators on Bernoulli functionals (Bernoulli annihilators, for short) admit much good operation properties with physical meanings.
For example, they together with their adjoint operators satisfy a canonical anti-commutation relation (CAR) in equal-time \cite{w-cl}.
In recent years, these operators have begun to find applications in developing a discrete-time stochastic calculus in
infinite dimensions.
Privault \cite{priv} used the Bernoulli annihilators to define the gradients for Bernoulli functionals in 2008.
In 2010, Nourdin \textit{et al} \cite{nourdin} investigated normal approximation of Rademacher functionals (a special case of Bernoulli functionals)
with the help of the Bernoulli annihilators.
Recently \cite{w-ch}, it has been shown that a wide class of quantum Markov semigroups can be constructed from the Bernoulli annihilators
and their adjoint operators.
As is known, quantum Markov semigroups are quantum analogues of the classical Markov semigroups in probability theory,
which provide a mathematical model for describing the irreversible evolution of quantum systems interacting
with the environment (see, e.g. \cite{ch-fa,meyer,partha}).

A Dirichlet form \cite{albe-1} in an $L^2$-space is a closed, positive, symmetric, densely defined bilinear form that has
the contraction property (also known as the Markov property).
As a mathematical structure, Dirichlet forms have close connections with many objects in probability theory,
quantum mechanics and quantum field theory (see \cite{albe-1} and references therein).
One typical example in this respect is that, under some mild conditions, a Dirichlet form determines a Markov process
through the Markov semigroup associated with it \cite{fuku}.
Thus Dirichlet forms play an important role in construction of Markov processes \cite{fuku}.
The classical Dirichlet forms \cite{beurling, fuku} are usually defined in the $L^2$-spaces of functions on locally compact metric spaces.
In the past three decades, however, Dirichlect forms on noncommutative and even infinite dimensional settings have also appeared successively.
For example, Albeverio and Hoegh-Krohn \cite{albe-2} constructed Dirichlet forms on $C^*$-algebras,
while Hida \textit{et al} \cite{hida} gave a Dirichlet form on white noise functionals by using Hida's differential operators.

In this paper, motivated by the work of Hida \textit{et al} \cite{hida},
we would like to apply the Bernoulli annihilators to constructing Dirichlet forms on Bernoulli functionals.
Our main work is as follows.

Let $\{\partial_k\}_{k\geq 0}$ be the Bernoulli annihilators. For a nonnegative function $w$ on $\mathbb{N}$,
we first define in the $L^2$-space of Bernoulli functionals a positive, symmetric, densely defined bilinear
form $\mathcal{E}_w$ in the following manner
\begin{equation*}
  \mathcal{E}_w(\xi,\eta) = \sum_{k=0}^{\infty}w(k)\langle \partial_k \xi, \partial_k \eta \rangle,
\end{equation*}
where $\langle\cdot,\cdot\rangle$ denotes the inner product in the $L^2$-space of Bernoulli functionals.
We then prove that $\mathcal{E}_w$ is closed and has the contraction property, hence it is a Dirichlet form.
We also obtain an operator representation of $\mathcal{E}_w$.
Finally, we consider an interesting semigroup of operators associated with $w$ on the $L^2$-space of Bernoulli functionals,
which we call the $w$-Ornstein-Uhlenbeck semigroup, and by using the Dirichlet form $\mathcal{E}_w$ we show
that the $w$-Ornstein-Uhlenbeck semigroup is a Markov semigroup.

We mention that the whole family of the annihilation and creation operators acting on Bernoulli functionals is interpreted
as a type of quantum Bernoulli noises \cite{w-cl, w-zh}.
As is seen, however, we only make use of the annihilation operators in the construction of our Dirichlet form.
Thus our work here does not mean that one can construct a Dirichlet form from quantum Bernoulli noises.
In fact, from a physical point of view, any Dirichlet forms cannot be constructed from any quantum Bernoulli noises.

{\bf Notation and conventions.} Throughout, $\mathbb{N}$ always denotes the set of all nonnegative integers.
We denote by $\Gamma$ the finite power set of $\mathbb{N}$, namely
\begin{equation}\label{eq-2-1}
    \Gamma
    = \{\,\sigma \mid \text{$\sigma \subset \mathbb{N}$ and $\#\,\sigma < \infty$} \,\},
\end{equation}
where $\#(\sigma)$ means the cardinality of $\sigma$ as a set.
Unless otherwise stated, letters like $j$, $k$ and $n$ stand for nonnegative integers, namely elements of $\mathbb{N}$.

\section{Bernoulli annihilators}\label{sec-2}

In this section, we briefly recall some necessary notions and facts about Bernoulli functionals and the annihilation operators on them.
For details, we refer to \cite{w-cl}.

Let $\Omega=\{-1,1\}^{\mathbb{N}}$ be the set of all mappings $\omega\colon \mathbb{N} \mapsto \{-1,1\}$, and
$(\zeta_n)_{n\geq 0}$ the sequence of canonical projections on $\Omega$ given by
\begin{equation}\label{eq-2-1}
    \zeta_n(\omega)=\omega(n),\quad \omega\in \Omega.
\end{equation}
Denote by $\mathscr{F}$ the $\sigma$-field on $\Omega$ generated by the sequence $(\zeta_n)_{n\geq 0}$.
Let $(p_n)_{n\geq 0}$ be a given sequence of positive numbers with the property that $0 < p_n < 1$ for all $n\geq 0$.
It is known \cite{priv} that there exists a unique probability measure $\mathbb{P}$ on $\mathscr{F}$ such that
\begin{equation}\label{eq-2-2}
\mathbb{P}\circ(\zeta_{n_1}, \zeta_{n_2}, \cdots, \zeta_{n_k})^{-1}\big\{(\epsilon_1, \epsilon_2, \cdots, \epsilon_k)\big\}
=\prod_{j=1}^k p_j^{\frac{1+\epsilon_j}{2}}(1-p_j)^{\frac{1-\epsilon_j}{2}}.
\end{equation}
for $n_j\in \mathbb{N}$, $\epsilon_j\in \{-1,1\}$ ($1\leq j \leq k$) with $n_i\neq n_j$ when $i\neq j$
and $k\in \mathbb{N}$ with $k\geq 1$. Thus we have a probability measure space $(\Omega, \mathscr{F}, \mathbb{P})$.

Let $Z=(Z_n)_{n\geq 0}$ be the sequence of random variables on $(\Omega, \mathscr{F}, \mathbb{P})$ defined by
\begin{equation}\label{eq-2-3}
   Z_n = \frac{\zeta_n + q_n - p_n}{2\sqrt{p_nq_n}},\quad n\geq0.
\end{equation}
where $q_n = 1-p_n$. Clearly $Z=(Z_n)_{n\geq 0}$ is an independent sequence of random variables on $(\Omega, \mathscr{F}, \mathbb{P})$,
and, for each $n\geq 0$, $Z_n$ has a distribution
\begin{equation}\label{eq-2-4}
  \mathbb{P}\{Z_n = \theta_n\}=p_n,\quad
    \mathbb{P}\{Z_n = -1/\theta_n\}=q_n,\quad n\geq 0
\end{equation}
with $\theta_n = \sqrt{q_n/p_n}$.
To be convenient, we set $\mathscr{F}_{-1}=\{\emptyset, \Omega\}$ and
\begin{equation*}
\mathscr{F}_n = \sigma(Z_k; 0\leq k \leq n),
\end{equation*}
the $\sigma$-field generated by $(Z_k)_{0\leq k \leq n}$ for $ n \geq 0$.
By convention $\mathbb{E}$ will denote the expectation with respect to $\mathbb{P}$.

Let $L^2(\Omega)$ be the space of square integrable random variables on $(\Omega, \mathscr{F}, \mathbb{P})$, namely
\begin{equation}\label{eq-2-5}
  L^2(\Omega) = L^2(\Omega, \mathscr{F}, \mathbb{P}).
\end{equation}
We denote by
$\langle\cdot,\cdot\rangle$ the usual inner product of the space $L^2(\Omega)$, and by $\|\cdot\|$ the corresponding norm.
It is known \cite{priv} that $Z$ has the chaotic representation property.
Thus $L^2(\Omega)$ has $\{Z_{\sigma}\mid \sigma \in \Gamma\}$ as its orthonormal basis,
where $Z_{\emptyset}=1$ and
\begin{equation}\label{eq-2-6}
    Z_{\sigma} = \prod_{j\in \sigma}Z_j,\quad \text{$\sigma \in \Gamma$, $\sigma \neq \emptyset$},
\end{equation}
which shows that $L^2(\Omega)$ is an infinite dimensional real Hilbert space.

\begin{remark}\label{rem-2-1}
It is easy to see that $\mathscr{F}=(Z_n; n \geq 0)$. Thus $\mathscr{F}$-measurable functions on $\Omega$ are usually known as
functionals of $Z$, or Bernoulli functionals simply. In particular, functions in $L^2(\Omega)$ are called square integrable
Bernoulli functionals.
\end{remark}

\begin{lemma}\label{lem-2-1}\cite{w-cl}
For each $k\in \mathbb{N}$, there exists a bounded operator $\partial_k$ on
$L^2(\Omega)$ such that
\begin{equation}\label{eq-2-7}
    \partial_k Z_{\sigma} = \mathbf{1}_{\sigma}(k)Z_{\sigma\setminus k},\quad
    \sigma \in \Gamma,
\end{equation}
where $\sigma\setminus k=\sigma\setminus \{k\}$
and $\mathbf{1}_{\sigma}(k)$ the indicator of $\sigma$ as a subset of $\mathbb{N}$.
\end{lemma}

\begin{lemma}\label{lem-2-2}\cite{w-cl}
Let $k\in \mathbb{N}$.  Then $\partial_k^{\ast}$, the adjoint of  operator $\partial_k$, has following property:
\begin{equation}\label{eq-2-8}
    \partial_k^{\ast} Z_{\sigma}
    = [1-\mathbf{1}_{\sigma}(k)]Z_{\sigma\cup k}\quad
    \sigma \in \Gamma,
\end{equation}
where $\sigma\cup k=\sigma\cup \{k\}$.
\end{lemma}

The operator $\partial_k$ and its adjoint $\partial_k^{\ast}$ are referred to as the annihilation operator and
creation operator at site $k$, respectively.
The next lemma shows that these operators satisfy a canonical anti-commutation relations (CAR) in equal-time.

\begin{lemma}\label{lem-2-3}\cite{w-cl}
Let $k$, $l\in \mathbb{N}$. Then it holds true that
\begin{equation}\label{eq-2-9}
    \partial_k \partial_l = \partial_l\partial_k,\quad
    \partial_k^{\ast} \partial_l^{\ast} = \partial_l^{\ast}\partial_k^{\ast},\quad
    \partial_k^{\ast} \partial_l = \partial_l\partial_k^{\ast}\quad (k\neq l)
\end{equation}
and
\begin{equation}\label{eq-2-10}
   \partial_k\partial_k= \partial_k^{\ast}\partial_k^{\ast}=0,\quad
   \partial_k\partial_k^{\ast} + \partial_k^{\ast}\partial_k=I,
\end{equation}
where $I$ is the identity operator on $L^2(\Omega)$.
\end{lemma}

\section{Forms constructed from Bernoulli annihilators}\label{sec-3}

In the present section, we show how to use the annihilation operators $\{\partial_k\}_{k \geq 0}$ to construct a closed,
positive, symmetric and densely defined bilinear form in $L^2(\Omega)$.

We first prove a convergence result concerning $\{\partial_k\}_{k \geq 0}$, which will play a
key role in our later discussions.

For a nonnegative function $w\colon \mathbb{N}\rightarrow \mathbb{R}$, we define the $w$-counting measure $\#_w(\cdot)$ as
\begin{equation}\label{eq-3-1}
  \#_w(\sigma) = \sum_{j\in \sigma}w(j),\quad \sigma\in \Gamma,
\end{equation}
where $\#_w(\sigma)=0$ if $\sigma=\emptyset$. Clearly $0\leq \#_w(\sigma) < \infty$ for all $\sigma\in \Gamma$.

\begin{theorem}\label{thr-3-1}
Let $w\colon \mathbb{N}\rightarrow \mathbb{R}$ be a nonnegative function. Define $\mathsf{D}_w$ as the linear subspace of $L^2(\Omega)$ given by
\begin{equation}\label{eq-3-2}
  \mathsf{D}_w= \Big\{\, \xi \in L^2(\Omega) \Bigm|\  \sum_{\sigma \in \Gamma} \#_w(\sigma) |\langle Z_{\sigma}, \xi\rangle|^2 < \infty\,\Big\}.
\end{equation}
Then $\mathsf{D}_w$ is a dense linear subspace of $L^2(\Omega)$, and moreover, for all $\xi$, $\eta \in \mathsf{D}_w$, the following series
is absolutely convergent:
\begin{equation}\label{eq-3-3}
  \sum_{k=0}^{\infty}w(k)\langle \partial_k\xi, \partial_k \eta\rangle.
\end{equation}
\end{theorem}

\begin{proof}
Clearly $\mathsf{D}_w$ is a linear subspace of $L^2(\Omega)$. On the other hand,
for each $\tau \in \Gamma$, in view of the fact that $\#_w(\tau) <\infty$, we have
\begin{equation*}
  \sum_{\sigma \in \Gamma} \#_w(\sigma) |\langle Z_{\sigma}, Z_{\tau}\rangle|^2 = \#_w(\tau)  <\infty.
\end{equation*}
Thus $\{Z_{\tau} \mid \tau\in \Gamma\} \subset \mathsf{D}_w$, which implies that
$\mathsf{D}_w$ is dense in $L^2(\Omega)$.

Now let $\xi$, $\eta \in \mathsf{D}_w$. Then,
for each $k\geq 0$, it follows from Lemma~\ref{lem-2-1} as well as the expansion $\xi = \sum_{\sigma \in \Gamma}\langle Z_{\sigma}, \xi\rangle Z_{\sigma}$  that
\begin{equation}\label{eq-3-4}
  \partial_k \xi
  = \sum_{\sigma \in \Gamma} \mathbf{1}_{\sigma}(k)\langle Z_{\sigma}, \xi\rangle Z_{\sigma\setminus k}.
\end{equation}
Thus, by a direct calculation, we have
\begin{equation*}
  \sum_{k=0}^{\infty}w(k)\|\partial_k \xi\|^2
  = \sum_{k=0}^{\infty}\sum_{\sigma \in \Gamma}\mathbf{1}_{\sigma}(k) w(k)|\langle Z_{\sigma}, \xi\rangle|^2
  = \sum_{\sigma \in \Gamma}\#_w(\sigma) |\langle Z_{\sigma}, \xi\rangle|^2
  <\infty.
\end{equation*}
With the same argument, we have
\begin{equation*}
  \sum_{k=0}^{\infty}w(k)\|\partial_k \eta\|^2
   <\infty.
\end{equation*}
It then follows from these two equalities that
\begin{equation*}
  \sum_{k=0}^{\infty} w(k)|\langle \partial_k\xi, \partial_k \eta\rangle|
    \leq \Big[\sum_{k=0}^{\infty}w(k)\|\partial_k \xi\|^2\Big]^{\frac{1}{2}}
       \Big[\sum_{k=0}^{\infty}w(k)\|\partial_k \eta\|^2\Big]^{\frac{1}{2}}
       <\infty,
\end{equation*}
which implies that the series $\sum_{k=0}^{\infty}w(k)\langle \partial_k\xi, \partial_k \eta\rangle$ is absolutely convergent.
\end{proof}

In view of the above theorem, we come naturally to the next definition, which introduces our main object of study.

\begin{definition}\label{def-3-1}
For a nonnegative function $w\colon \mathbb{N}\rightarrow \mathbb{R}$, we define $\mathcal{E}_w$ as
\begin{equation}\label{eq-3-5}
  \mathcal{E}_w(\xi, \eta) =\sum_{k=0}^{\infty}w(k)\langle \partial_k\xi, \partial_k \eta\rangle,\quad \xi,\, \eta \in \mathsf{D}_w,
\end{equation}
and call $(\mathcal{E}_w,\mathsf{D}_w)$ the $w$-energy form.
\end{definition}

It is easy to see that $(\mathcal{E}_w,\mathsf{D}_w)$ is a positive, symmetric and densely defined bilinear form in $L^2(\Omega)$.
Note that if we take $w(k)\equiv 1$, then we get
\begin{equation*}
  \mathcal{E}_w(\xi, \xi) =\int_{\Omega}\|\nabla\xi(\omega)\|_{l^2(\mathbb{N})}^2d\mathbb{P}(\omega),\quad \xi\in \mathsf{D}_w,
\end{equation*}
where $\nabla$ denotes the gradient operator, which is defined on $\mathsf{D}_w$ and valued in $L^2(\Omega\!\times\! \mathbb{N})$, the space of
jointly square integrable stochastic processes on $(\Omega, \mathscr{F}, \mathbb{P})$ (see \cite{w-lc} for details).
This justifies the name of $(\mathcal{E}_w,\mathsf{D}_w)$.

In order to examine basic properties of the $w$-energy form $(\mathcal{E}_w,\mathsf{D}_w)$,
we introduce another bilinear form $\widehat{\mathcal{E}}_w$ on $\mathsf{D}_w$ as
\begin{equation}\label{eq-3-6}
  \widehat{\mathcal{E}}_w(\xi, \eta) = \mathcal{E}_w(\xi, \eta) + \langle \xi,\eta\rangle,\quad \xi,\, \eta \in \mathsf{D}_w,
\end{equation}
where $\langle\cdot,\cdot\rangle$ is the inner product of $L^2(\Omega)$.
It is then easy to see that $\widehat{\mathcal{E}}_w$ is again an inner product on $\mathsf{D}_w$. We denote by $\|\cdot\|_{\widehat{\mathcal{E}}_w}$ the norm induced by $\widehat{\mathcal{E}}_w$.
A direct calculation gives that
\begin{equation}\label{eq-3-7}
  \|\xi\|_{\widehat{\mathcal{E}}_w}^2= \sum_{\sigma \in \Gamma}[\#_w(\sigma) +1]|\langle Z_{\sigma}, \xi\rangle|^2,\quad \xi\in \mathsf{D}_w,
\end{equation}
where $\#_w(\cdot)$ is the $w$-counting measure as defined by (\ref{eq-3-1}).

\begin{theorem}\label{thr-3-2}
Let $w\colon \mathbb{N}\rightarrow \mathbb{R}$ be a nonnegative function. Then the $w$-energy form $(\mathcal{E}_w,\mathsf{D}_w)$ is closed,
namely $(\mathsf{D}_w,\widehat{\mathcal{E}}_w)$ is a Hilbert space.
\end{theorem}

\begin{proof}
We need only to show that $\mathsf{D}_w$ is complete with respect to norm $\|\cdot\|_{\widehat{\mathcal{E}}_w}$.

Let $\{\xi_n\}_{n\geq 1}\subset \mathsf{D}_w$ be a Cauchy sequence with respect to norm $\|\cdot\|_{\widehat{\mathcal{E}}_w}$.
Then it is also a Cauchy sequence with respect to norm $\|\cdot\|$ since $\|\cdot\|\leq \|\cdot\|_{\widehat{\mathcal{E}}_w}$.
Thus there exists $\xi \in L^2(\Omega)$ such that $\|\xi_n -\xi\|\rightarrow 0$ as $n\rightarrow \infty$. Now we show that
$\xi \in \mathsf{D}_w$ and  $\|\xi_n -\xi\|_{\widehat{\mathcal{E}}_w}\rightarrow 0$ as $n\rightarrow \infty$.

Let $\epsilon > 0$. Then, by the property that $\{\xi_n\}_{n\geq 1}$ is a Cauchy sequence with respect to $\|\cdot\|_{\widehat{\mathcal{E}}_w}$, we
know that there exists a positive integer $K\geq 1$ such that
\begin{equation}\label{eq-3-8}
 \sum_{\sigma \in \Gamma}[\#_w(\sigma) +1]|\langle Z_{\sigma}, \xi_m -\xi_n\rangle|^2 =\|\xi_m -\xi_n\|_{\widehat{\mathcal{E}}_w}^2
  < \epsilon^2,\quad \forall\, m,\, n > K.
\end{equation}
For all $n>K$, since
\begin{equation*}
  \lim_{m\to \infty}[\#_w(\sigma) +1]|\langle Z_{\sigma}, \xi_m -\xi_n\rangle|^2 = [\#_w(\sigma) +1]|\langle Z_{\sigma}, \xi -\xi_n\rangle|^2,\quad
  \forall\,\sigma \in \Gamma,
\end{equation*}
it follows by applying the well-known Fatou's Lemma \cite{reed} to (\ref{eq-3-8}) that
\begin{equation}\label{eq-3-9}
 \sum_{\sigma \in \Gamma}[\#_w(\sigma) +1]|\langle Z_{\sigma}, \xi -\xi_n\rangle|^2\leq \epsilon^2,
\end{equation}
which implies that $\xi-\xi_n \in \mathsf{D}_w$, hence $\xi \in \mathsf{D}_w$ and
\begin{equation*}
\|\xi-\xi_n \|_{\widehat{\mathcal{E}}_w} = \Big\{\sum_{\sigma \in \Gamma}[\#_w(\sigma) +1]|\langle Z_{\sigma}, \xi -\xi_n\rangle|^2\Big\}^{1/2}\leq \epsilon.
\end{equation*}
Therefore $\|\xi_n -\xi\|_{\widehat{\mathcal{E}}_w}\rightarrow 0$ as $n\rightarrow \infty$, namely $\{\xi_n\}$ converges in $\mathsf{D}_w$ with
respect to norm $\|\cdot\|_{\widehat{\mathcal{E}}_w}$.
\end{proof}

\section{Contraction property}\label{sec-4}

Let $w\colon \mathbb{N}\rightarrow \mathbb{R}$ be a given nonnegative function.
In this section, we further prove that the $w$-energy form $(\mathcal{E}_w,\mathsf{D}_w)$ has the contraction property, hence it is a Dirichlet form.

We first make some necessary preparations. Let $\mathsf{S}=\mathrm{Span}\{Z_{\sigma}\mid \sigma\in \Gamma\}$, the linear subspace of $L^2(\Omega)$ spanned by
the system $\{Z_{\sigma}\mid \sigma\in \Gamma\}$, which is obviously a dense linear subspace of $L^2(\Omega)$.
For $n\geq 0$, we denote by $\mathsf{S}_n$ the linear subspace of all $\mathscr{F}_n$-measurable random variables in $L^2(\Omega)$.
It can be verified that $\mathsf{S}_n$ is a $2^{n+1}$-dimensional closed subspace of $L^2(\Omega)$, and has an orthonormal basis $\{Z_{\sigma}\mid \sigma\in \Gamma\!_{n]}\}$, where $\Gamma_{n]}=\{\sigma\in \Gamma\mid \max \sigma \leq n\}$. And moreover, these subspaces have the following relation
\begin{equation*}
\bigcup_{n=0}^{\infty}\mathsf{S}_n = \mathsf{S}.
\end{equation*}

Recall that elements of $\Omega$ are mappings $\omega\colon \mathbb{N}\rightarrow \{-1, 1\}$. For $k\geq 0$ and
$\omega \in \Omega$, we can naturally define two mappings $\omega^+_k$, $\omega^-_k \colon \mathbb{N}\rightarrow \{-1, 1\}$ as
\begin{equation}\label{eq-4-1}
\begin{array}{ll}
  \omega^+_k(n)=
\left\{
  \begin{array}{ll}
    1, & \hbox{$n=k$;} \\
    \omega(n), & \hbox{$n\neq k$,}
  \end{array}
\right.&
\omega^-_k(n)=
\left\{
  \begin{array}{ll}
    -1, & \hbox{$n=k$;} \\
    \omega(n), & \hbox{$n\neq k$,}
  \end{array}
\right.
\end{array}
\end{equation}
which, of course, remain elements of $\Omega$.

The following proposition is a slight variant of a result given in \cite{priv}, which shows that the annihilation operator $\partial_k$ acts just like a difference operator.

\begin{lemma}\label{lem-4-1}\cite{priv}
Let $k\geq 0$ and $\xi \in \mathsf{S}$. Then
\begin{equation}\label{eq-4-2}
  \partial_k \xi(\omega) = \sqrt{p_kq_k}\,\left[\xi(\omega^+_k)-\xi(\omega^-_k)\right],\quad  \omega\in\Omega.
\end{equation}
\end{lemma}

\begin{proof}
We need only to show that (\ref{eq-4-2}) holds for each $\xi = Z_{\sigma}$ with $\sigma\in \Gamma$.
Let $\sigma\in \Gamma$. Then it follows from (\ref{eq-2-1}), (\ref{eq-2-3}) and (\ref{eq-4-1}) that
\begin{equation}\label{eq-4-3}
  Z_j(\omega) =  Z_j(\omega^+_k)=Z_j(\omega^-_k),\quad  \omega\in\Omega,\, j\neq k,
\end{equation}
and
\begin{equation}\label{eq-4-4}
  Z_k(\omega^+_k)=\sqrt{\frac{q_k}{p_k}},\quad Z_k(\omega^-_k)=-\sqrt{\frac{p_k}{q_k}}.
\end{equation}
Thus, in case of $k\notin \sigma$, $Z_{\sigma}(\omega) =  Z_{\sigma}(\omega^+_k)=Z_{\sigma}(\omega^-_k)$ for $\omega\in\Omega$, which implies
\begin{equation*}
  \sqrt{p_kq_k}\,\left[Z_{\sigma}(\omega^+_k)  - Z_{\sigma}(\omega^-_k)\right]=0, \quad\omega\in \Omega,
\end{equation*}
which, together the formula $\partial_kZ_{\sigma}= \mathbf{1}_{\sigma}(k)Z_{\sigma\setminus k}$, gives
\begin{equation*}
   \partial_k Z_{\sigma}(\omega) = \sqrt{p_kq_k}\,\left[Z_{\sigma}(\omega^+_k)-Z_{\sigma}(\omega^-_k)\right],\quad  \omega\in\Omega.
\end{equation*}
Now if $k\in \sigma$, then
$Z_{\sigma\setminus k}(\omega) =  Z_{\sigma\setminus k}(\omega^+_k)=Z_{\sigma\setminus k}(\omega^-_k)$ for $\omega\in\Omega$,
which together with (\ref{eq-4-4}) yields
\begin{equation*}
\begin{split}
Z_{\sigma\setminus k}(\omega)
   &=\sqrt{p_kq_k}\,\left[Z_{\sigma\setminus k}(\omega^+_k)\sqrt{\frac{q_k}{p_k}}  + Z_{\sigma\setminus k}(\omega^-_k)\sqrt{\frac{p_k}{q_k}}\right]\\
    &=\sqrt{p_kq_k}\,\left[Z_{\sigma\setminus k}(\omega^+_k)Z_k(\omega^+_k)  - Z_{\sigma\setminus k}(\omega^-_k)Z_k(\omega^-_k)\right]\\
    &=\sqrt{p_kq_k}\,\left[Z_{\sigma}(\omega^+_k)  - Z_{\sigma}(\omega^-_k)\right],\quad \omega\in\Omega,
\end{split}
\end{equation*}
which, together the formula $\partial_kZ_{\sigma}= \mathbf{1}_{\sigma}(k)Z_{\sigma\setminus k}$, still leads to
\begin{equation*}
  \partial_k Z_{\sigma}(\omega) = \sqrt{p_kq_k}\,\left[Z_{\sigma}(\omega^+_k)-Z_{\sigma}(\omega^-_k)\right],\quad \omega\in\Omega.
\end{equation*}
This completes the proof.
\end{proof}

\begin{definition}\label{def-4-1}
A contraction function is a function $C\colon \mathbb{R}\rightarrow \mathbb{R}$ with
$C(0)=0$, and $|C(s)-C(t)|\leq |s-t|$ for all $s$, $t\in \mathbb{R}$.
\end{definition}

As usual, we use $C\circ\xi$ to mean the composition of a contraction function $C$ and a random variable $\xi$.
The next theorem then shows that $\mathsf{S}$ is invariant under the action of contraction functions.

\begin{theorem}\label{thr-4-2}
Let $C$ be a contraction function.
Then $C\circ\xi \in \mathsf{S}$ for all $\xi \in \mathsf{S}$.
\end{theorem}

\begin{proof}
Let $\xi \in \mathsf{S}$. Then there exists some $n\geq 0$ such that $\xi \in \mathsf{S}_n$, which implies
$C\circ\xi \in \mathsf{S}_n$ since $C$ is a continuous function. Thus $C\circ\xi \in \mathsf{S}$.
\end{proof}

\begin{theorem}\label{thr-4-3}
Let $\xi \in L^2(\Omega)$ and $C$ a contraction function.
Then for all $k\geq 0$ it holds that
\begin{equation}\label{eq-4-5}
\|\partial_k(C\circ\xi)\| \leq \|\partial_k\xi\|.
\end{equation}
\end{theorem}

\begin{proof}
Let $k\geq0$. Then we can take a sequence $(\xi_n)_{n\geq 1}$ in $\mathsf{S}$ such that $\|\xi_n -\xi\|\rightarrow 0$ as $n\rightarrow \infty$.
It follows from Theorem~\ref{thr-4-2} that $\{C\circ\xi_n\}_{n\geq 1} \subset \mathsf{S}$. Thus,
for each $n\geq 1$,
by using Lemma~\ref{lem-4-1} and the contraction property of $C$, we have
\begin{equation*}
\begin{split}
  |\partial_k(C\circ\xi_n)(\omega)|
   &=  \sqrt{p_kq_k}\,|C\circ\xi_n(\omega^+_k)-C\circ\xi_n(\omega^-_k)|\\
   &\leq \sqrt{p_kq_k}\,|\xi_n(\omega^+_k)-\xi_n(\omega^-_k)|\\
   &= |\partial_k\xi_n(\omega)|
\end{split}
\end{equation*}
with $\omega\in \Omega$, which implies
\begin{equation*}
\|\partial_k(C\circ\xi_n)\|\leq \|\partial_k\xi_n\|,\quad  n \geq 1.
\end{equation*}
It is easy to see that
$\|C\circ\xi_n-C\circ\xi\|\rightarrow 0$ as $n\rightarrow \infty$.
Thus, by letting $n\rightarrow \infty$ in the above inequality, we finally come to (\ref{eq-4-5}).
\end{proof}

\begin{theorem}\label{thr-4-4}
The $w$-energy form $(\mathcal{E}_w,\mathsf{D}_w)$ has the contraction property, namely,
for all contraction function $C\colon \mathbb{R}\rightarrow \mathbb{R}$ and all $\xi\in \mathsf{D}_w$,
it holds that $C\circ\xi\in \mathsf{D}_w$ and
\begin{equation}\label{eq-4-6}
  \mathcal{E}_w(C\circ\xi,C\circ\xi) \leq \mathcal{E}_w(\xi,\xi).
\end{equation}
\end{theorem}

\begin{proof}
Let $\xi\in \mathsf{D}_w$ and $C\colon \mathbb{R}\rightarrow \mathbb{R}$ a contraction function. Then by Theorem~\ref{thr-4-3} we have
\begin{equation*}
\begin{split}
\sum_{\sigma \in \Gamma}\#_w(\sigma) |\langle Z_{\sigma}, C\circ\xi\rangle|^2
  &= \sum_{k=0}^{\infty}w(k)\|\partial_k (C\circ \xi)\|^2\\
  &\leq \sum_{k=0}^{\infty}w(k)\|\partial_k \xi\|^2\\
  &= \sum_{\sigma \in \Gamma}\#_w(\sigma) |\langle Z_{\sigma}, \xi\rangle|^2,
\end{split}
\end{equation*}
which, together with the assumption $\xi\in \mathsf{D}_w$, implies $C\circ\xi\in \mathsf{D}_w$.
It then follows from the definition of $\mathcal{E}_w$ that
\begin{equation*}
  \mathcal{E}_w(C\circ\xi,C\circ\xi)
= \sum_{k=0}^{\infty}w(k)\|\partial_k(C\circ\xi)\|^2
\leq \sum_{k=0}^{\infty}w(k)\|\partial_k\xi\|^2
=\mathcal{E}_w(\xi,\xi).
\end{equation*}
\end{proof}

In the literature, a closed, positive, symmetric, densely defined bilinear form in an $L^2$-space is called a Dirichlet form if
it additionally has the contraction property (see, e.g.  \cite{albe-1,fuku, hida}).
Summing up our discussions above, we actually arrive at the next important conclusion.

\begin{corollary}\label{coro-4-5}
The $w$-energy form $(\mathcal{E}_w,\mathsf{D}_w)$ is a Dirichlet form in $L^2(\Omega)$.
\end{corollary}

\section{Application}\label{sec-5}

Markov semigroups are semigroups of contraction operators that are closely related to Markov processes in probability theory.
In the final section, we consider an interesting class of semigroups of contraction operators on $L^2(\Omega)$.
As an application of our results in the previous sections, we will prove that this class of semigroups are
actually Markov ones.

Let $w\colon \mathbb{N}\rightarrow \mathbb{R}$ be a nonnegative function. For
each $t\geq 0$, one can define a contraction operator $P^w_t$ on $L^2(\Omega)$ as
\begin{equation}\label{eq-5-1}
  P^w_t\xi = \sum_{\sigma \in \Gamma} e^{-t\#_w(\sigma)} \langle Z_{\sigma}, \xi\rangle Z_{\sigma},\quad \xi \in L^2(\Omega),
\end{equation}
where $\#_w(\cdot)$ is the $w$-counting measure defined by (\ref{eq-3-1}).
One can verify that the family $P^w=(P^w_t)_{t\geq 0}$ forms a strongly continuous semigroup of contraction
operators on $L^2(\Omega)$.

\begin{definition}\label{def-5-1}
The semigroup $P^w=(P^w_t)_{t\geq 0}$ defined by (\ref{eq-5-1}) is called the $w$-Ornstein-Uhlenbeck semigroup on $L^2(\Omega)$.
\end{definition}

We note that, with $w(k)\equiv 1$, the $w$-Ornstein-Uhlenbeck semigroup becomes the usual Ornstein-Uhlenbeck semigroup
considered in \cite{priv, w-lc}.

\begin{theorem}\label{thr-5-1}
Let $w\colon \mathbb{N}\rightarrow \mathbb{R}$ be a nonnegative function. Then the $w$-Ornstein-Uhlenbeck semigroup
$P^w =(P^w_t)_{t\geq 0}$ is a Markov semigroup on $L^2(\Omega)$, namely for all $t\geq 0$ it holds
that
\begin{equation}\label{eq-5-2}
 0\leq P^w_t\xi \leq 1\quad  \text{$\mathbb{P}$-a.e.}
\end{equation}
whenever $\xi\in L^2(\Omega)$ with $0\leq \xi \leq 1$ $\mathbb{P}$-a.e.
\end{theorem}

\begin{proof}
Consider the operator $N_w$ in $L^2(\Omega)$ given by
\begin{equation}\label{eq-5-3}
  N_w\xi = \sum_{\sigma \in \Gamma} \#_w(\sigma) \langle Z_{\sigma},\xi\rangle Z_{\sigma},\quad \xi \in \mathrm{Dom}\,N_w,
\end{equation}
where $\mathrm{Dom}\,N_w$ denotes the domain of the operator $N_w$, which is defined as
\begin{equation*}
\mathrm{Dom}\,N_w =\Big\{\, \xi \in L^2(\Omega) \Bigm| \sum_{\sigma \in \Gamma} [\#_w(\sigma)]^2 |\langle Z_{\sigma},\xi\rangle |^2 <\infty \,\Big\},
\end{equation*}
where $\#_w(\cdot)$ is the $w$-counting measure defined by (\ref{eq-3-1}).
It is easy to show that $N_w$ is a positive, self-adjoint, densely defined operator in $L^2(\Omega)$.
And moreover, by a direct calculation, we find that
\begin{equation}\label{eq-5-4}
   P^w_t = e^{-tN_w},\quad t\geq 0,
\end{equation}
where $e^{-tN_w}$ means the spectral integral of the function $s \mapsto e^{-ts}$ with respect to the spectral
measure of $N_w$. Thus $-N_w$ is exactly the infinitesimal generator of the semigroup
$P^w =(P^w_t)_{t\geq 0}$.

We now verify that $N_w$ and $\mathcal{E}_w$ share the following relationship
\begin{equation}\label{eq-5-5}
  \mathcal{E}_w(\xi, \eta) = \langle \xi,N_w\eta\rangle,\quad \xi\in \mathsf{D}_w,\, \eta \in  \mathrm{Dom}\,N_w
\end{equation}
with $\mathrm{Dom} N_w \subset \mathsf{D}_w$.
In fact, if $\eta \in \mathrm{Dom}\,N_w$, then by the well known Cauchy inequality we have
\begin{equation*}
  \sum_{\sigma\in \Gamma}\#_w(\sigma)|\langle Z_{\sigma}, \eta\rangle|^2
  \leq  \left\{\sum_{\sigma\in \Gamma}[\#_w(\sigma)]^2|\langle Z_{\sigma}, \eta\rangle|^2\right\}^{\frac{1}{2}}
        \left\{\sum_{\sigma\in \Gamma}|\langle Z_{\sigma}, \eta\rangle|^2\right\}^{\frac{1}{2}}
  <\infty,
\end{equation*}
which implies $\eta \in \mathsf{D}_w$. Thus $\mathrm{Dom}\,N_w \subset \mathsf{D}_w$.
Let $\xi\in \mathsf{D}_w$ and $\eta \in  \mathrm{Dom}\,N_w$. Then, by (\ref{eq-3-4}) and (\ref{eq-3-5}), we have
\begin{equation*}
\begin{split}
  \mathcal{E}_w(\xi, \eta)
    &=\sum_{k=0}^{\infty}w(k)\langle \partial_k\xi, \partial_k \eta\rangle\\
    &=  \sum_{k=0}^{\infty}w(k)\sum_{\sigma \in \Gamma} \mathbf{1}_{\sigma}(k)\langle Z_{\sigma}, \xi\rangle\langle Z_{\sigma}, \eta\rangle\\
    &= \sum_{\sigma \in \Gamma}\#_w(\sigma) \langle Z_{\sigma}, \xi\rangle\langle Z_{\sigma}, \eta\rangle.
\end{split}
\end{equation*}
Here we take use of the absolute convergence of the following series
\begin{equation*}
  \sum_{\sigma \in \Gamma} \sum_{k=0}^{\infty} \mathbf{1}_{\sigma}(k)w(k)\langle Z_{\sigma}, \xi\rangle\langle Z_{\sigma}, \eta\rangle,
\end{equation*}
which can be verified straightforwardly.
On the other hand, by the definition of $N_w$ and the expansion $\xi = \sum_{\sigma \in \Gamma}\langle Z_{\sigma}, \xi\rangle Z_{\sigma}$, we have
\begin{equation*}
  \langle \xi,N_w\eta\rangle  = \sum_{\sigma \in \Gamma} \#_w(\sigma)\langle Z_{\sigma}, \xi\rangle\langle Z_{\sigma}, \eta\rangle.
\end{equation*}
Thus $\mathcal{E}_w(\xi, \eta) = \langle \xi,N_w\eta\rangle$.

Combining (\ref{eq-5-4}), (\ref{eq-5-5}) and Corollary~\ref{coro-4-5} with the general theory of Dirichlet forms \cite{albe-1}, we finally
know that
$P^w =(P^w_t)_{t\geq 0}$ is a Markov semigroup.
\end{proof}

\section{Discussion}\label{sec-6}

Let $X$ be a locally compact separable metric space and $\mu$ be a positive Radon measure on $X$.
For a Dirichlet form $\mathcal{E}$ in $L^2(X, \mu)$, one thing researchers are interested in is
whether it is regular \cite{fuku}.
Indeed, if $\mathcal{E}$ is a regular Dirichlet form in $L^2(X, \mu)$, then
there is a symmetric Hunt process $\mathrm{M}$ with $\mathcal{E}$ as its Dirichlet form.
If, additionally, $\mathcal{E}$ possesses the local property, then the process $\mathrm{M}$
is even a diffusion process \cite{fuku}.
In that case, the process $\mathrm{M}$ is called the diffusion process associated with the Dirichlet form $\mathcal{E}$.
It is well known that diffusion processes play an important role in many problems in mathematical physics.

As is shown, in our case $\mathcal{E}_w$ is a Dirichlet form in $L^2(\Omega) \equiv L^2(\Omega, \mathbb{P})$
for a positive function $w\colon \mathbb{N} \rightarrow \mathbb{R}$.
On the other hand, being endowed with an appropriate metric, $\Omega$ can become a compact separable metric space and $\mathbb{P}$ can become
a positive Radon measure on $\Omega$.
It is then natural to consider whether or not $\mathcal{E}_w$ is regular , and whether or not $\mathcal{E}_w$ possesses the local property.
Such problems, however, are far from being simple and might be related to deeper topological properties of $\Omega$.

\section*{Acknowledgement}

This work is supported by National Natural Science Foundation of China (Grant No. 11461061).

\end{document}